\documentclass{article}

\newcommand{\negeer}[1]{}
\def\notintp{\mathrel{/\kern-0.6em\rhd}} 
\usepackage{amssymb}
\usepackage{amsfonts}
\usepackage{latexsym}
\usepackage{epsfig}
\usepackage{psfig}
\usepackage{xspace}
\usepackage{color}
\usepackage{amstext}
\usepackage{rotating}
\usepackage{amsmath}
\usepackage{import}
\usepackage{proof}	
\usepackage[english,dutch]{babel}
\usepackage{makeidx}

\newcommand{\qedsymbol}{$\dashv$}

\newcommand{\il}{{\ensuremath{\textup{\textbf{IL}}}}\xspace}
\newcommand{\extil}[1]{\ensuremath{\textup{\textbf{IL}}{\sf\ensuremath{#1}}}\xspace}

\newcommand{\intl}[1]{{\ensuremath {\textup{\textbf{IL}}}({\rm #1})}}

\newcommand{\ilm}{\extil{M}}
\newcommand{\ilp}{\extil{P}}
\newcommand{\ilw}{\extil{W}}
\newcommand{\ilr}{\extil{R}}

\newcommand{\ilwstar}{\extil{W^*}}
\newcommand{\ilrstar}{\extil{R^*}}
\newcommand{\ilwstarpenul}{\extil{P_0W^*}}

\newcommand{\ilal}{\intl{All}\xspace}
\newcommand{\formal}[1]{\ensuremath{{\sf {#1}}\xspace}}

\newcommand{\prop}{\formal{Prop}\xspace}

\newcommand{\principle}[1]{\formal{#1}}
\renewcommand{\principle}[1]{\formal{#1}}
\newcommand{\principel}[1]{\formal{#1}}

\newcommand{\eqbydef}{:=}

\newcommand{\formil}{\ensuremath{{\sf Form}_{\il}} \xspace}

\newcommand{\mo}{\principle{M_0}\xspace}

\usepackage{amsmath,amsthm}

\usepackage{amssymb,stmaryrd,color,bm}

\newtheorem{theorem}{Theorem}[section]

\newtheorem{definition}[theorem]{Definition}
\newtheorem{lemma}[theorem]{Lemma}
\newtheorem{corollary}[theorem]{Corollary}

\title{
A New Principle in the Interpretability Logic of all Reasonable Arithmetical Theories}
\author{Evan Goris\\ and\\ Joost J. Joosten}
\date{2011}

\begin{document}
\maketitle

\begin{abstract}
The interpretability logic of a mathematical theory describes the structural behavior of interpretations over that theory. Different theories have different logics. This paper revolves around the question what logic describes the behavior that is present in all theories with a minimum amount of arithmetic; the intersection over all such theories so to say. We denote this target logic by \ilal.

In this paper we present a  new principle \principle{R} in \ilal. We show that \principle{R} does not follow from the logic \ilwstarpenul that contains all previously known principles. This is done by providing a modal incompleteness proof of \ilwstarpenul : showing that \principle{R} follows semantically but not syntactically from \ilwstarpenul. Apart from giving the incompleteness proof by elementary methods, we also sketch how to work with so-called Generalized Veltman Semantics as to establish incompleteness. To this extent, a new version of this Generalized Veltman Semantics is defined and studied. Moreover, for the important principles the frame correspondences are calculated.

After the modal results it is shown that the new principle \principle{R} is indeed valid in any arithmetically theory. The proof employs some elementary results on definable cuts in arithmetical theories.
\end{abstract}


\section{Introduction}
Interpretations of one theory or structure into another are omnipresent in (meta-) mathematical practice. Interpretability logics study the structural behavior of interpretations. Below we shall provide precise definitions. The structural behavior of interpretations is different for different kind of theories thus yielding different interpretability logics.

For example, for finitely axiomatized theories, the corresponding logic turned out to be \ilp as defined below. For theories like Peano Arithmetic with full induction\footnote{Technically speaking the property of so-called essential reflexivity is sufficient. A theory is essentially reflexive if any of its finite extensions proves the consistency of any finite sub-theory thereof. } the interpretability logic is \ilm. It is a long standing open problem what the core-structural behavior of interpretations is. That is, what is the interpretability logic that is valid in any sort of theory. It turns out that ``any sort of theory'' is a bit too ample a quantification. As we shall see below, we at least need some core amount of arithmetic as to do coding of syntax. Therefore the question is often paraphrased as: What is the interpretability logic of all reasonable arithmetical theories? We denote this logic by \ilal.

This paper revolves around this question and presents some major results. In \cite{joo:prol00} a conjecture was posed that $\ilal = \extil{W^*P_0}$. In this paper we refute this conjecture by exposing a new arithmetically valid principle.
We shall prove the modal incompleteness of the logic
\ilwstarpenul by introducing a new principle \principel{R}. Next, we show that the principle \principel{R} follows semantically from \ilwstarpenul but is not provable in \ilwstarpenul. 

We shall expose two proof methods here to prove that 
\principel{R} does not follow syntactically from \ilwstarpenul. 
The first, in Section \ref{subs:generalizedinmain}, develops some general theory for proving incompleteness via so-called \emph{Generalized Veltman Semantics}. However, it will turn out that the frame condition for the principle $W$ is so ghastly and cumbersome that however possible to work with, proofs become too involved. 

The second proof method,
in Section \ref{subs:incompletepenulwestar}, uses the regular Veltman semantics and some sort of bisimulation argument and a full proof is given that \principel{R} does not follow syntactically from \ilwstarpenul. 

To conclude, we shall prove in Section \ref{subs:risvetsound} that actually \principel{R} is sound in any reasonable arithmetical theory. In particular, this implies that \ilal can not be \ilwstarpenul.

We found the principle \principel{R} by trying to formulate a sufficient condition for the logic \ilwstarpenul to be modally complete. We think that this illustrates nicely that a modal formulation of an arithmetical phenomenon can be very useful to obtain new arithmetical results.

\section{Preliminaries}

This paper is the third and final in a series of three. All definitions and motivations behind the definitions were already included in \cite{jogo:mm08}. For completeness and readability we include the main definitions and issues also in this paper. 

\subsection{Arithmetic and an upper bound to \ilal}

As with (almost) all interesting occurrences of modal logic, 
interpretability logics are used to study a hard mathematical notion.
Interpretability logics, as their name slightly suggests, are 
used to study the notion of formal interpretability. In this subsection
we shall very briefly say what this notion is and how modal logic is 
used to study it.

We are interested in first order theories in the language of arithmetic. All 
theories we will consider will thus be arithmetical theories. 
Moreover, we want our theories to have a certain minimal strength.
That is, they should contain a certain core theory, say 
${\sf I}\Delta_{0} + \Omega_1$ from \cite{HP}. This will allow
us to do reasonable coding of syntax. We call these theories reasonable
arithmetical theories.

Once we can code syntax, we can write down a decidable predicate 
${\sf Proof}_T(p,\varphi)$ that holds on the standard model
precisely when $p$ is a $T$-proof of $\varphi$.\footnote{We take
the liberty to not make a distinction between a syntactical object
and its code.} We get a provability predicate by quantifying
existentially, that is, 
${\sf Prov}_T (\varphi ) \eqbydef \exists p \ {\sf Proof}_T (p, \varphi )$.

We can use these coding techniques to code the notion of
formal interpretability too.
Roughly, a theory $U$ interprets a theory $V$ if there is some
sort of translation so that every theorem of $V$ is under that
translation also a theorem of $U$.

\begin{definition}
Let $U$ and $V$ be reasonable arithmetical theories.
An interpretation $j$ from $V$ in $U$ is a pair
$\langle \delta , F \rangle$. Here, $\delta$ is called a 
domain specifier. It is a formula with one free variable.
The $F$ is a map that sends an $n$-ary relation symbol of
$V$ to a formula of $U$ with $n$ free variables. 
(We treat functions and constants as relations with additional properties.)
The interpretation $j$ induces a translation from formulas 
$\varphi$ of $V$ to
formulas $\varphi^j$ of $U$ by replacing relation symbols by their
corresponding formulas and by relativizing quantifiers to $\delta$.
We have the following requirements.
\begin{itemize}
\item
$(R (\vec{x}))^j = F(R) (\vec{x})$

\item
The translation induced by $j$ commutes with the boolean 
connectives. Thus, for example,
$(\varphi \vee \psi)^j = \varphi^j \vee \psi^j$.
In particular
$(\bot)^j=(\vee_{\varnothing})^j = \vee_{\varnothing}=\bot$

\item
$(\forall x \ \varphi)^j= \forall x \ (\delta (x) \rightarrow \varphi^j)$

\item
$V\vdash \varphi \Rightarrow U\vdash \varphi^j$

\end{itemize}
We say that $V$ is interpretable in $U$ if there exists an interpretation
$j$ of $V$ in $U$.
\end{definition}

Using the ${\sf Prov}_T (\varphi)$ predicate, it is possible
to code the notion of formal interpretability in arithmetical
theories. This gives rise to a formula ${\sf Int}_T (\varphi , \psi )$,
to hold on the standard model precisely when $T+ \psi$
is interpretable in $T+\varphi$. This formula is related to the
modal part by means of arithmetical realizations.
\medskip

The modal language of interpretability logics is the same as that of provability logics but now augmented by a binary modality $\rhd$ to denote interpretability.Thus, we define the interpretability formulas as 
\[
\formil \eqbydef \bot \mid \prop \mid (\formil \rightarrow \formil) \mid (\Box \formil)
\mid (\formil \rhd \formil) 
\]

Here \prop is a countable set of propositional 
variables $p,q,r,s,t,p_0,p_1,\ldots$.
We employ the usual definitions of the logical operators
$\neg , \vee , \wedge$ and $\leftrightarrow$. Also shall we write 
$\Diamond \varphi$ for $\neg \Box \neg \varphi$. We refer the reader to \cite{jogo:mm08} for more details and standard reading conventions.
\medskip

Now we can define the link between the modal language and the arithmetical counterpart. 
\begin{definition}
An arithmetical realization $*$ is a mapping that assigns
to each propositional variable an arithmetical sentence. This mapping
is extended to all modal formulas in the following way.
\begin{itemize}
\item[-]
$(\varphi \vee \psi)^* = \varphi^* \vee \psi^*$ and likewise for 
other boolean connectives. In particular
$\bot^* = (\vee_{\varnothing})^*=\vee_{\varnothing}=\bot$.

\item[-]
$(\Box \varphi)^*= {\sf Prov}_T (\varphi^*)$

\item[-]
$(\varphi \rhd \psi)^*= {Int}_T (\varphi^*,\psi^*)$
\end{itemize}
\end{definition}
\noindent
From now on, the $*$ will always range over realizations.
Often we will write $\Box_T \varphi$ instead of ${\sf Prov}_T (\varphi)$
or just even $\Box \varphi$. The $\Box$ can thus denote both a 
modal symbol and an arithmetical formula. For the $\rhd$-modality we
adopt a similar convention. We are confident that no
confusion will arise from this.

\begin{definition}
An interpretability principle of a theory $T$ is a modal formula
$\varphi$ that is provable in $T$ under any realization. That is,
$\forall *\ T\vdash \varphi^*$. The interpretability logic of a
theory $T$, we write \intl{T}, is the set of
all interpretability principles.
\end{definition}

For two classes of theories, \intl{T} is known.

\begin{definition}
A theory $T$ is reflexive if it proves the consistency of any of 
its finite subtheories. It is essentially reflexive if any finite
extension of it is reflexive.
\end{definition}

\begin{theorem}[Berarducci \cite{bera:inte90}, Shavrukov \cite{shav:logi88}]\label{theo:shav}
If $T$ is an essentially reflexive theory, then $\intl{T}=\ilm$.
\end{theorem}

\begin{theorem}[Visser \cite{viss:inte90}]
If $T$ is finitely axiomatizable, then $\intl{T}=\ilp$.
\end{theorem}

Now we have all in place to define our central subject of interest.

\begin{definition}\label{defi:ilall}
The interpretability logic of all reasonable arithmetical theories, 
we write
\intl{All}, is the set of formulas $\varphi$ such that
$\forall T  \, \forall *\ T\vdash \varphi^*$. Here the $T$ ranges over
all the reasonable arithmetical theories.
\end{definition}

For sure \intl{All} should be in the intersection of \ilm and \ilp.
Up to now, \intl{All} is unknown. In \cite{joo:prol00} it is conjectured
to be \extil{P_0W^*}. It is one of the major open problems in the 
field of interpretability logics, to characterize \intl{All} in a modal
way. As $\ilp \cap \ilm$ is known to be a strict upper bound, we also know some strict lower bounds. We close off this section on preliminaries by defining these lower bounds and providing modal semantics for interpretability logics. 

\subsection{A lower bound to \ilal}

We first define a core logic which will be part of all interpretability logics studied.
\begin{definition}\label{defi:il}
The logic \il is the smallest set of formulas being closed under
the rules of Necessitation\footnote{That is, from $\varphi$ you are allowed to conclude $\Box \varphi$.} and of Modus Ponens, that contains
all tautological formulas and all instantiations of the following
axiom schemata.

\begin{enumerate}
\item[${\sf L1}$]\label{ilax:l1}
        $\Box(A\rightarrow B)\rightarrow(\Box A\rightarrow\Box B)$
\item[${\sf L2}$]\label{ilax:l2}
        $\Box A\rightarrow \Box\Box A$
\item[${\sf L3}$]\label{ilax:l3}
        $\Box(\Box A\rightarrow A)\rightarrow\Box A$
\item[${\sf J1}$]\label{ilax:j1}
        $\Box(A\rightarrow B)\rightarrow A\rhd B$
\item[${\sf J2}$]\label{ilax:j2}
        $(A\rhd B)\wedge (B\rhd C)\rightarrow A\rhd C$
\item[${\sf J3}$]\label{ilax:j3}
        $(A\rhd C)\wedge (B\rhd C)\rightarrow A\vee B\rhd C$
\item[${\sf J4}$]\label{ilax:j4}
        $A\rhd B\rightarrow(\Diamond A\rightarrow \Diamond B)$
\item[${\sf J5}$]\label{ilax:j5}
        $\Diamond A\rhd A$
\end{enumerate}
\end{definition}
We will write $\il \vdash \varphi$ for $\varphi \in \il$. An \il-derivation
or \il-proof of $\varphi$ is a finite sequence of formulae ending on 
$\varphi$, each being a logical tautology, an instantiation of one
of the axiom schemata of \il, or the result of applying either Modus Ponens or
Necessitation to formulas earlier in the sequence.

Apart from the axiom schemata exposed in Definition \ref{defi:il} we will
need consider other axiom schemata too.

\begin{enumerate}
\item[${\sf M}$] 
$A \rhd B \rightarrow A \wedge \Box C \rhd B \wedge \Box C$

\item[${\sf P}$]
$A \rhd B \rightarrow \Box (A \rhd B)$

\item[${\sf M_0}$]
$A \rhd B \rightarrow \Diamond A \wedge \Box C \rhd B \wedge \Box C$

\item[${\sf W}$]
$A \rhd B \rightarrow A \rhd B \wedge \Box \neg A$

\item[${\sf W^*}$]
$A \rhd B \rightarrow B \wedge \Box C \rhd B \wedge \Box C \wedge \Box \neg A$

\item[${\sf P_0}$]
$A \rhd \Diamond B \rightarrow \Box (A \rhd B)$

\item[${\sf R}$]
$A\rhd B \rightarrow \neg (A \rhd \neg C) \rhd B \wedge \Box C$

\end{enumerate}

If $\sf X$ is a set of axiom schemata we will denote by \extil{X} the
logic that arises by adding the axiom schemata in $\sf X$ to \il.

Now, with the results of this paper, we know that \ilwstarpenul is a strict lower bound for \ilal.

\subsection{Semantics}

Interpretability logics come with a Kripke-like semantics.
As the signature of our language is countable, we shall only consider 
countable models.

\begin{definition}\label{defi:frames}
An \il-frame is a triple $\langle W,R,S \rangle$. Here $W$ is a non-empty countable
universe,
$R$ is a binary relation on $W$ and $S$ is a set of binary relations on $W$, indexed 
by elements of $W$.
The $R$ and $S$ satisfy the following requirements.
\begin{enumerate}
\item
$R$ is conversely well-founded\footnote{A relation $R$ on $W$ is called 
conversely well-founded if every non-empty subset of $W$ has an $R$-maximal element.}

\item
$xRy \ \& \ yRz \rightarrow xRz$

\item
$yS_xz \rightarrow xRy \ \& \ xRz$

\item
$xRy \rightarrow yS_x y$

\item
$xRyRz \rightarrow yS_x z$ \label{defi:point:inclusion}

\item
$uS_x v S_x w \rightarrow u S_x w$

\end{enumerate}
\end{definition}
\il-frames are sometimes also called Veltman frames.
We will on occasion speak of $R$ or $S_x$ transitions instead of relations.
If we write $ySz$, we shall mean that $yS_xz$ for some $x$. 
$W$ is sometimes called the universe, or domain, of the frame and its elements
are referred to as worlds or nodes. With $x{\upharpoonright}$ we shall denote
the set $\{ y \in W \mid x Ry \}$.
We will often represent $S$ by a ternary relation in the canonical way, writing
$\langle x,y,z\rangle$ for $yS_xz$.

\begin{definition}
An \il-model is a quadruple 
$\langle W, R , S, \Vdash \rangle$. Here $\langle W, R , S, \rangle$ is 
an \il-frame and $\Vdash$ is a subset of $W \times \prop$. We write
$w \Vdash p$ for $\langle w,p\rangle \in \ \Vdash$.
As usual, $\Vdash$ is extended to a subset $\widetilde{\Vdash}$ of 
$W \times \formil$ by demanding the following.
\begin{itemize}

\item
$w \widetilde{\Vdash} p$ iff $w \Vdash p$ for $p \in \prop$

\item
$w \not \widetilde{\Vdash} \bot$

\item
$w \widetilde{\Vdash} A \rightarrow B$ iff $w \not \widetilde{\Vdash} A$ or 
$w \widetilde{\Vdash} B$

\item
$w \widetilde{\Vdash} \Box A$ iff 
$\forall v \ (wRv \Rightarrow v \widetilde{\Vdash} A)$

\item
$w \widetilde{\Vdash} A \rhd B$ iff 
$\forall u \ (w R u \wedge u\widetilde{\Vdash} A
\Rightarrow \exists v \ (u S_w v  \widetilde{\Vdash} B))$

\end{itemize}
\end{definition}

Note that $\widetilde{\Vdash}$ is completely determined by $\Vdash$.
Thus we will denote $\widetilde{\Vdash}$ also by $\Vdash$.






%





\section{Generalized semantics}\label{subs:generalizedinmain}
\newcommand{\ilset}{\ensuremath{\il_{\textrm{set}}}\xspace}
\newcommand{\ilxset}[1]{\ilset{\principle{#1}}}
\newcommand{\po}{\principel{P_0}}
\newcommand{\ilmoset}{\ilxset{\mo}\xspace}
\newcommand{\ilposet}{\ilxset{\po}\xspace}
\newcommand{\ilrset}{\ilxset{{\sf R}}\xspace}


In \cite{Svej91}, \v{S}vejdar showed the independence of some extensions of \il.
Some of these logics, however, had the same class of characteristic Veltman frames.
Naturally, frames alone are not sufficient to distinguish between such logics so
\v{S}vejdar used models combined with some bisimulation arguments instead.
A generalized Veltman semantics, intended to uniformize this method, was proposed by de Jongh.
This generalized semantics was previously investigated by Vukovic\'c \cite{Vuk96}, Joosten \cite{joo98}
and Verbrugge and was successfully used to show independence of certain extensions of \il. 

We will set both the generalized Veltman semantics and the model/bisimulation method
to work in order to distinguish some extensions of \il, which are indistinguishable using Veltman
frames alone.
We use a slight variation of the semantics used in \cite{Vuk96}.
Any result in this section can be obtained with the old semantics, we think that
nevertheless this might be a useful variation.

\begin{definition}[\ilset-frame]\label{defi:ilsetframe}
A structure $\langle W,R,S \rangle$ is an \ilset-frame iff.
\begin{enumerate}
\item\label{ils:i} $W$ is an non-empty set.
\item\label{ils:ii} $R$ is a transitive and conversely well-founded binary
        relation on $W$.
\item\label{ils:iii} $S\subseteq W\times W\times(\mathcal{P}(W)-\{\emptyset\})$,
			such that (where we write $yS_xY$ for $(x,y,Y)\in S$)
	{
	\renewcommand{\theenumii}{(\alph{enumii})}
	\renewcommand{\labelenumii}{\theenumii}
	\begin{enumerate}
	\item\label{ils:iiia} if $xS_wY$ then $wRx$ and for all $y\in Y$, $wRy$,
	\item\label{ils:iiib} $S$ is quasi-reflexive: $wRx$ implies $xS_w\{x\}$,
	\item\label{ils:iiic} $S$ is quasi-transitive: If $xS_wY$ then for all $y\in Y$
		we have that if $y\not\in Z$ and 
	        $yS_wZ$ then $xS_wZ$,
	\item\label{ils:iiid} $wRxRy$ implies $xS_w\{y\}$.
	\end{enumerate}
	}
\end{enumerate}
\end{definition}

\begin{definition}[\ilset-model]
An \ilset-model is a structure $\langle W,R,S,\Vdash\rangle$ such
that $\langle W,R,S\rangle$ is an \ilset-frame and $\Vdash$ is a
binary relation between elements of $W$ and modal formulas such
that the following cases apply.
\begin{enumerate}
\item $\Vdash$ commutes with boolean connectives.
        For instance, $w\Vdash A\wedge B$ iff. $w\Vdash A$
        and $w\Vdash B$.
\item $w\Vdash\Box A$ iff. for all $x$ such that $wRx$ we have that  $x\Vdash A$.
\item $w\Vdash A\rhd B$ iff. for all $x$ such that $wRx$ and
        $x\Vdash A$ there exists some $Y$, such that $xS_wY$ and for all
        $y\in Y$, $y\Vdash B$.
\end{enumerate}
For $\ilset$-models $F=\langle W,R,S,\Vdash\rangle$ and $Y\subseteq W$ we will write
$Y\Vdash A$ for $\forall y\in Y,\;y\Vdash A$.

\end{definition}

As usual, we say that a formula $A$ is valid on an $\ilset$-frame
$F=\langle W,R,S\rangle$ if for any model $\overline
F=\langle W,R,S,\Vdash\rangle$, based on $F$, and any $w\in W$, we
have $\overline F,w\Vdash A$.

\begin{lemma}[Soundness of \il]
If $\il\vdash A$ then for any $\ilset$-frame $F$, $F\models A$.
\end{lemma}

\begin{proof}
Validity is preserved under modus ponens and generalization and
trivially any propositional tautology is valid on each \ilset-frame.
So it is enough to show
that all axioms of \il are valid on each $\ilset$-frame.
We only treat ${\sf J}2$: $(A\rhd B)\wedge(B\rhd C)\rightarrow A\rhd C$.

Suppose $w\Vdash A\rhd B$ and $w\Vdash B\rhd C$. Pick some $x$
    with $wRx$ and suppose $x\Vdash A$.
    There exists some $Y$ with $xS_wY$ and $Y\Vdash B$.
    W.l.o.g.\ we can assume that for some $y\in Y$, $y\not\Vdash C$.
    Fix such a $y$. Since $y\Vdash B$ and $wRy$ there
    exists some $Z$ such that $yS_wZ$ and $Z\Vdash C$.
    In particular, $y\not\in Z$.
    And thus we have $xS_wZ$.
\end{proof}

\begin{theorem}[Completeness of \il]
If $A$ is valid on each \ilset-frame, then $\il\vdash A$.
\end{theorem}
\begin{proof}
Suppose $\il\not\vdash A$. Then there exists an \il-model $M=\langle W,R,S\rangle$,
and some $m\in M$ such that $M,m\Vdash\neg A$. Let
\mbox{$M'=\langle W,R,S',\Vdash'\rangle$}, where
$\Vdash'=\Vdash$ on propositional variables and is extended as usual, and
\[S'=\{(w,x,Y)\mid \forall y\in Y\; xS_wy \}.\]
It is easy to see that $M'$ is an \ilset-model.
As an example let us see that $S$ is quasi-transitive.
Suppose $xS'_wX$, $y\in X$ and $yS'_wY$. (We can assume $y\not\in Y$, but we won't use this.)
Pick $y'\in Y$. Then $xS_wy$ and $yS_wy'$. Thus $xS_wy'$.
Since $y'\in Y$ was arbitrary we conclude $xS'_wY$.

A straightforward induction on $B$ shows that for all $B$ we have $w \Vdash' B\Leftrightarrow w\Vdash B$.
Thus we have $m\Vdash'\neg A$ and in particular $A$ is not valid on the underling frame
of $M'$.
\end{proof}

\begin{definition}[\ilmoset-frame]\label{definition:ilmosetframe}
An $\ilset$-frame is an \ilmoset-frame iff. for all $w,x,y,Y$ such that
$wRxRyS_wY$ there exists some $Y'\subseteq Y$ such that
\begin{enumerate}
\item $xS_wY'$ and
\item for all $y'\in Y'$ we have that for all $z$, $y'Rz\rightarrow xRz$.
\end{enumerate}
\end{definition}

\begin{lemma}\label{lemm:moframes}
For any $\ilset$-frame $F=\langle W,R,S\rangle$ we have $F\models\mo$ iff. $F$ is an \ilmoset-frame.
\end{lemma}

\begin{proof}

($\Leftarrow$) Suppose $F$ is an \ilmoset-frame.
Let $\overline F=\langle W,R,S,\Vdash\rangle$ be a model based on this frame.
Pick $w\in W$ and suppose $w\Vdash A\rhd B$.
Pick $x\in W$ with $wRx$ and $x\Vdash\Diamond A\wedge\Box C$.
Now there exists some $y$ with $xRy$ and $y\Vdash A$.
Thus, for some $Y$, $yS_wY$ and $Y\Vdash B$.
Since $F$ is an \ilmoset-frame, there exists some $Y'\subseteq Y$ such that $xS_wY'$ and
for all $y'\in Y'$ we have that for all $z$, $y'Rz\rightarrow xRz$.
So, in particular, $Y'\Vdash\Box C$.

($\Rightarrow$) Suppose $F\models\mo$. Choose $w,x,y,Y$ such that $wRxRyS_wY$.
Let $p,q,s$ be distinct proposition variables.
Define an $\ilset$-model $\overline F=\langle W,R,S,\Vdash\rangle$ as follows.
\begin{align*}
v&\Vdash p\Leftrightarrow v=y\\
v&\Vdash q\Leftrightarrow v\in Y\\
v&\Vdash s\Leftrightarrow xRv
\end{align*}
Now, $w\Vdash p\rhd q$ and thus $w\Vdash \Diamond p\wedge\Box s\rhd q\wedge\Box s$.
Also, $x\Vdash\Diamond p\wedge\Box s$.
So, there exists some $Y'$ such that $xS_wY'$ and $Y'\Vdash q\wedge\Box s$.
But the only candidates for such an $Y'$ are the subsets of $Y$.
Also, since $Y'\Vdash\Box s$, by definition of $\Vdash$ we have
$y'\in Y'$ and $y'Rz$ implies $xRz$.
\end{proof}

\begin{definition}[\ilposet-frame]\label{definition:ilmposetframe}
An $\ilset$-frame is an \ilposet-frame iff. for all $w,x,y,Y,Z$
such that
\begin{enumerate}
\item $wRxRyS_wY$ and
\item for all $y\in Y$ there exists some $z\in Z$ with $yRz$,
\end{enumerate}
we have that there exists some $Z'\subseteq Z$ with $yS_xZ'$.
\end{definition}

\begin{lemma}\label{lemm:poframes}
For any $\ilset$-frame $F=\langle W,R,S\rangle$ we have $F\models\po$ iff. $F$ is an
\ilposet-frame.
\end{lemma}

\begin{proof}

($\Leftarrow$) Suppose $F$ is an \ilposet-frame.
And let $\overline F=\langle W,R,S,\Vdash\rangle$ be an $\ilset$-model based
on this frame.
Let $w\in W$ and suppose $w\Vdash A\rhd\Diamond B$. Pick $x,y$ in $W$
with $wRxRy$ and $y\Vdash A$. There exists some $Y$ with $yS_wY$ and $Y\Vdash\Diamond B$.
Put $Z=\{z\mid z\Vdash B\}$. Now for all $y\in Y$ there exists some $z\in Z$ such that $yRz$.
So, there exists some $Z'\subseteq Z$ with $yS_xZ'$.

($\Rightarrow$) Suppose $F\models\po$.
Choose $w,x,y\in W$ and $Y,Z\subseteq W$ such that
$wRxRyS_wY$ and for all $y\in Y$ there exists some $z\in Z$ with $yRz$.
Let $p,q$ be distinct propositional variables.
Define the $\ilset$-model $\overline F=\langle W,R,S,\Vdash\rangle$ as follows.
\begin{align*}
v&\Vdash p\Leftrightarrow v=y\\
v&\Vdash q\Leftrightarrow v\in Z
\end{align*}
Now, $Y\Vdash\Diamond q$.
So, $w\Vdash p\rhd\Diamond q$ and thus, since $w\Vdash\po$, $w\Vdash\Box (p\rhd q)$.
So for some $Z'$ we have $yS_xZ$ and $Z'\Vdash q$. But the only candidates for such
$Z'$ are the subsets of $Z$.
\end{proof}

\begin{figure}
\begin{center}
\resizebox{4.5cm}{!}{\input{ponotmo.xfig.pstex_t}}
\end{center}
\caption{An \ilposet-frame which is not an \ilmoset-frame.}
\label{figu:ponotmo}
\end{figure}

\begin{lemma}\label{lemm:ponotmo}
There exists an \ilposet-frame which is not an \ilmoset-frame.
\end{lemma}
\begin{proof}
Consider Figure \ref{figu:ponotmo}. It represents an \ilset-frame.
For clarity we have omitted the following arrows.
Those needed for the transitivity of $R$.
Those needed for the quasi-reflexivity of $S$.
Those needed for the inclusion of $S$ in $R$.
Additionally, quasi-transitivity dictates that we need $xS_w\{z_2\}$, $yS_w\{z_1\}$ and $yS_w\{z_2\}$.
All the other ones are drawn.

Let us first see that we actually have an \ilposet-frame.
So suppose $vRaRbS_vB$.
And let $Z$ be such that for all $b'\in B$ there exists some $z\in Z$ such that $b'Rz$.
It is not hard to see that only for $v=w$, $a=x$, $b=y$ and $B=\{y_1,y_2\}$ such a $Z$
exists. And that moreover this $Z$ must equal $\{z_1,z_2\}$.
According to the \po-condition we must find a $Z'\subseteq Z$ such that $yS_xZ$.
And $\{z_1\}$ is such a $Z'$.

Now let us see that we do not have an \ilmoset-frame.
Put $Y=\{y_1,y_2\}$. We have $wRxRyS_wY$.
So, if we do have an \ilmoset-frame then for some $Y'\subseteq Y$ we have
$xS_wY'$ and for all $y'\in Y'$ we have that for all $z$, $y'Rz$ implies $xRz$.
But the only $Y'\subseteq Y$ for which $xS_wY'$ is $Y$ itself.
We have $y_2\in Y$, $y_2Rz_2$ but not $xRz_2$.
\end{proof}

\begin{theorem}\label{theorem:moNotFromPo}
$\extil{P_0}\not\vdash\mo$.
\end{theorem}
\begin{proof}
If $\extil{P_0}\vdash\mo$ then $\mo$ is valid on any \ilposet-frame.
But then any \ilposet-frame is an \ilmoset-frame.
Which, by Lemma \ref{lemm:ponotmo} is not so.
\end{proof}

\begin{corollary}
$\extil{P_0}\not\vdash\principle{R}$.
\end{corollary}

\begin{proof}
By Theorem \ref{theorem:moNotFromPo} and the fact that \principle{M_0} follows from \principle{R}.
\end{proof}

\begin{definition}
Let $F=\langle W,R,S\rangle$ be an \ilset-frame.
For any $wRx$ we say that $\Gamma\subseteq W$ is a choice set
for $(w,x)$ iff. for all $X$ such that $xS_wX$, $X\cap\Gamma\neq\emptyset$.
\end{definition}

\begin{definition}
Let $F=\langle W,R,S\rangle$ be an \ilset-frame.
We say that $F$ is an $\ilset \sf{R}$-frame iff.
$wRxRyS_wY$ implies that for all choice sets $\Gamma$ for $(x,y)$ there exists
some $Y'=Y'(\Gamma)\subseteq Y$ such that $xS_wY'$
and for all $y'\in Y'$ we have that for all $z$, $y'Rz$ implies $z\in\Gamma$.
\end{definition}

\begin{lemma}
An \ilset-frame $F=\langle W,R,S\rangle$ is an $\ilset\sf{R}$-frame iff. $F\models{\sf R}$.
\end{lemma}
\begin{proof}
$(\Rightarrow)$ Suppose $F$ is an $\ilset\sf{R}$-frame.
Let $\overline F=\langle W,R,S,\Vdash\rangle$ be a model based on $F$.
Choose $w,x\in W$ and suppose $wRx$, $w\Vdash A\rhd B$ and $x\Vdash \neg(A\rhd C)$.
We have to find some $Y'$ with $xS_wY'$ and $Y'\Vdash B\wedge\Box \neg C$.
There exists some $y\in W$ such that $xRy$, $y\Vdash A$ and for all $U$ such that $yS_xU$
there exists some $u\in U$ with $u\Vdash \neg C$. Let $\Gamma$ be a choice set for $(x,y)$
such that $\Gamma\Vdash\neg C$ and $\Gamma\subseteq\bigcup_{yS_xU} U$.
Since $w\Vdash A\rhd B$ we can find some $Y$ such that $yS_wY$ and $Y\Vdash B$.
By the ${\sf R}$ frame condition we can find some $Y'\subseteq Y$ such that
$xS_wY'$ and for all $y'\in Y'$ we have that for all $z$, $y'Rz$ implies $z\in\Gamma$.
So since $\Gamma\Vdash \neg C$ we conclude $Y'\Vdash B\wedge\Box\neg C$.

$(\Leftarrow)$. Suppose $F\models{\sf R}$. Let $w,x,y,Y\in W$ and suppose $wRxRyS_wY$.
Let $\Gamma$ be a choice set for $(x,y)$.
Let $p,q,s,t$ be distinct propositional variables.
Define the \ilset-model $\overline F=\langle W,R,S,\Vdash\rangle$ as follows.
\begin{align*}
v\Vdash p &\Leftrightarrow v=y\\
v\Vdash q &\Leftrightarrow v\in Y\\
v\Vdash s &\Leftrightarrow v\not\in\Gamma
\end{align*}
Now, $w\Vdash p\rhd q$. So, $w\Vdash \neg(p\rhd s)\rhd q\wedge\Box\neg s$.
Also, $\Gamma\Vdash\neg s$. So, $x\Vdash\neg(p\rhd s)$
and therefore there exists some $Y'$ such that $xS_wY'$ and $Y'\Vdash q\wedge\Box s$.
Since $Y'\Vdash q$ we must have $Y'\subseteq Y$.
Now let $y'\in Y'$ and pick some $z$ for which $y'Rz$.
Then $z\Vdash \neg s$ and thus
by definition of $\Vdash$, $z\in\Gamma$.
\end{proof}

\begin{theorem}
$\extil{P_0M_0} \nvdash \principle{R}$.
\end{theorem}

\begin{proof}
Now that we have all the frame conditions at hand, we will provide a frame that is both a \principle{P_0} and a \principle{M_0} frame but not an \principle{R} frame. We define the required frame $F$ as follows.
\[
\begin{array}{l}
F=\langle W, R, S\rangle \  \mbox{with } W = \{  w, x, y, a_0, a_1, b_0, b_1\}\\
R \mbox{ and $S$ are defined by the minimal requirements below:} \\
\ \\
wRxRyS_wA\\
\ \ \ \ \ \, \ \ \ yS_xB\\
\mbox{where $A = \{a_0,a_1\}$ and $B= \{  b_0, b_1\}$}\\
a_i R b_i \mbox{ for $i=0,1$} \mbox{ and }\\ 
x R b_i \mbox{ for $i=0,1$}.
\end{array}
\]
We conclude the proof with a series of easy observations.
\begin{enumerate}
\item
$F$ is an $\ilxset{M_0}$ frame is clear: let $Y'=Y=A$ in Definition \ref{definition:ilmosetframe}.

\item
$F$ is an $\ilxset{P_0}$ frame is clear: let $Z'=Z=B$ in Definition \ref{definition:ilmposetframe}.

\item
$F$ is not an $\extil{R}$ frame: Let $\Gamma$ be a choice set for $(x,y)$ that omits $b_0$. As for any $Y'$ we have $yS_wY' \subseteq Y$ implies $Y'=Y$, we see that $a_0\in Y'$. But $a_0Rb_0$ and $b_0\notin \Gamma$.

\end{enumerate}
\end{proof}

We can also formulate a frame condition for \principle{W}. However it shall turn out that this frame condition becomes so intricate that it is not efficient to work with even over finite frames. However, one can check that the exposed counter frames above are indeed also $\ilxset{W}$ frames. We choose not to do so and rather give direct proofs that include \principle{W} in Section \ref{subs:incompletepenulwestar}. We start by defining a higher order property $\sf Not{-}W$ on frames.
In this definition, capital letters shall range over subsets and lower case to elements of the domain. The index $i$ is supposed to run over the natural numbers.

\begin{definition}
\[
\begin{array}{ll}
{\sf Not{-}W} \ \ :=\ \  & \exists \,  w,\  z_0,\  \{Y_i\}_{i\in \omega}, \ \{ y_i\}_{i\in \omega,\  y_i\in Y_i} ,\  Z,\   
\{  z_{i+1}\}_{i \in \omega, \ z_{i+1}\in Z}\\
\ & [ \forall i \in \omega (z_i S_w Y_i \ni y_i Rz_{i+i}) \ \& \\
\ & \ \forall z\in Z \exists i\in \omega zS_w Y_i \ \& \\
\ & \ \forall z\in Z \forall Y\ (zS_wY \wedge Y \subseteq (\cup_{i\in \omega}Y_i) \ \to \  
\exists z' \in Z \exists y \in Y yRz')]\\

\end{array}
\]
\end{definition}

\begin{lemma}
For any $\ilset$ frame $F$ we have that
\[
F \models {\sf Not{-}W} \ \ \Leftrightarrow \ \ F \not \models W.
\]
\end{lemma}

\begin{proof}
"$\Rightarrow$": Suppose $\sf Not{-}W$ holds. We use the same notation as in the definition and set out to define a valuation so that the instance $p\rhd q \to p \rhd q \wedge \Box \neg p$ of $W$ fails. 

We define
\[
\begin{array}{lll}
a\Vdash p & \Leftrightarrow & a \in Z\\
a \Vdash q& \Leftrightarrow & \exists i \in \omega \ a \in Y_i\\
\end{array}
\]
Now clearly $w\Vdash p\rhd q$ as $\forall z\in Z \exists i\in \omega zS_w Y_i$ and $p$ is only true at points in $Z$ and the $Y_i$ make $q$ true. However, $w\Vdash p\rhd q \wedge \Box \neg p$ can never hold. For, suppose that some $Y$ and some $z\in Z$ we have that $zS_wY$ and $Y\Vdash q$. By the definition of $\Vdash$ clearly, $Y \subseteq (\cup_{i\in \omega}Y_i)$, whence $\exists z' \in Z \exists y \in Y yRz')$ and $Y\not \Vdash \Box \neg p$ as $z'\Vdash p$.

"$\Leftarrow$": Suppose that $A\rhd B \to A\rhd B \wedge \Box \neg A$ fails to hold in $w$ in some model based on $F$. We will set out to find  the required $z_0,\  \{Y_i\}_{i\in \omega}, \ \{ y_i\}_{i\in \omega,\  y_i\in Y_i} ,\  Z,\   \{  z_{i+1}\}_{i \in \omega, \ z_{i+1}\in Z}$.

As $w \Vdash \neg (A\rhd B \wedge \Box \neg A)$, we can find $z_0$ with $wRz_0\Vdash A$ such that for no $Y$ with $z_0 S_w Y$ we have $Y\Vdash B\wedge \Box \neg A$. To find our other entities, we will need a technical definition of $\mathcal{R}_{w,z_0}$ of those those worlds that are \emph{reachable} from w by means of $R$ and $S_w$ successors.
\[
\begin{array}{l}
z_0 \in \mathcal{R}_{w,z_0};\\
y \in \mathcal{R}_{w,z_0} \ \& \ y S_wY \ \Rightarrow Y \subseteq \mathcal{R}_{w,z_0}.
\end{array}
\]
Next, we define 
\[
Z := \{ x \in \mathcal{R}_{w,z_0} \mid x \Vdash A \wedge (\exists Y \exists y \in Y( z_0 S_wY \wedge yRx)\vee x=z_0)
\}
\]
and
\[
\mathbb{Y} := \{  Y \subseteq \mathcal{R}_{w,z_0} \mid z S_w Y \Vdash B \mbox{ for some } z\in Z \}.
\]
Now it is easy to pick $\{ Y_i\}_{i\in \omega}$ with $Y_i \in \mathbb{Y}$ and to pick $\{z_i\}_{i\in \omega} \subseteq Z$ such that $z_i S_w Y_i \ni y_i R z_{i+1}$: as the $z_i \Vdash A$ we can go via $S_w$ to some $Y\Vdash B$ whence by definition $Y \in \mathbb{Y}$; as $Y\not \Vdash \Box \neg A$, at some $R$ successor $\tilde{z}$ of some $y\in Y$ we have $\tilde{z}\Vdash A$ whence $\tilde{z}\in Z$.

By the definitions of $Z$ and the $Y_i$ we have that $ \forall z\in Z \exists i\in \omega zS_w Y_i $. Thus we only need to check that $\forall z\in Z \forall Y\ (zS_wY \wedge Y \subseteq (\cup_{i\in \omega}Y_i) \ \to \  
\exists z' \in Z \exists y \in Y yRz')$. But this is also not hard. If we consider any $z \in Z$ and $Y$ for which $zS_wY \wedge Y \subseteq (\cup_{i\in \omega}Y_i)$, we see that $Y\Vdash B$. But, as also $z_0S_wY$, we need to have $Y \not \Vdash B \wedge \Box \neg A$ whence we can find some $y\in Y$ and $z' \in Z$ with $yRz'\Vdash B$.
\end{proof}

\section{Incompleteness of \extil{P_0W^*}}\label{subs:incompletepenulwestar}

Let us first calculate the frame condition of \principel{R} where 
\[
{\sf R} \eqbydef A\rhd B \rightarrow \neg (A \rhd \neg C)\rhd B \wedge \Box C.
\]
It turns out to be the same frame condition as for \principel{P_0}
(see \cite{joo98}).

\begin{lemma}
$F\models \principel{R} \Leftrightarrow [xRyRzS_xuRv\rightarrow zS_yv]$
\end{lemma}

\begin{proof}
``$\Leftarrow$'' Suppose that at some world $x\Vdash A\rhd B$.
We are to show $x\Vdash \neg (A \rhd \neg C)\rhd B \wedge \Box C$. Thus,
if $xRy\Vdash \neg (A \rhd \neg C)$ we need to go via an $S_x$ to a 
$u$ with $u\vdash B \wedge \Box C$.

As $y\Vdash \neg (A \rhd \neg C)$, we can find $z$ with $yRz\Vdash A$.
Now, by $x\Vdash A\rhd B$, we can find $u$ with $yS_xu\Vdash B$. We shall now see that
$u \Vdash B \wedge \Box C$. 
For, if $uRv$, then by our assumption, $zS_yv$, and by 
$y\Vdash \neg (A \rhd \neg C)$, we must have $v\vdash C$. Thus, 
$u\Vdash B\wedge \Box C$ and clearly $yS_xu$.

``$\Rightarrow$'' We suppose that \principel{R} holds. Now we consider 
arbitrary $a,b,c,d$ and $e$ with $aRbRcS_adRe$. For propositional 
variables $p,q$ and $r$ we define a valuation $\Vdash$ as follows.
\[
\begin{array}{lll}
x\Vdash p & :\Leftrightarrow & x=c \\
x\Vdash q& :\Leftrightarrow & x=d \\
x\Vdash r& :\Leftrightarrow & cS_bx \\
\end{array}
\]
Clearly, $a\Vdash p\rhd q$ and $b\Vdash \neg (p\rhd \neg r)$.
By \principel{R} we conclude $a\Vdash \neg (p\rhd \neg r)\rhd q \wedge \Box r$.
Thus, $d\Vdash q\wedge \Box r$ which implies $cS_be$.
\end{proof}

As $\sf P_0$ and $\sf R$ have the same frame condition we can never find an $\sf R$-frame on which $\sf P_0$ fails to hold. However, the following elementary lemma tells us that it is not necessary to work with frames.

\begin{lemma}\label{lemm:modelsound}
Let $M$ be a model such that $\forall\,  w{\in}M \ $ 
$w \Vdash \extil{X}$ then $\extil{X} \vdash \varphi \Rightarrow M\models \varphi$.
\end{lemma}

\begin{proof}
By induction on the derivation of $\varphi$.
\end{proof}

We can now prove the main theorem of this section.

\begin{theorem}
$\ilwstarpenul \nvdash \principel{R}$
\end{theorem}

\begin{figure}
\begin{center}
\input{countermodel.xfig.pstex_t}
\end{center}
\caption{\ilwstarpenul is incomplete}\label{figu:incompleteilwp}
\end{figure}

\begin{proof}
%
%

We consider the model $M$ from Figure \ref{figu:incompleteilwp}
and shall see that
$M \models \ilwstarpenul$ but $M,a \not \Vdash \principel{R}$.
By Lemma \ref{lemm:modelsound} we conclude that 
$\ilwstarpenul \nvdash \principel{R}$.

As $M$ satisfies the frame condition for \principel{W^*}, it is clear
that $M\models \principel{W^*}$. We shall now see that
$M\models A \rhd \Diamond B \rightarrow \Box (A\rhd B)$ for any 
formulas $A$ and $B$.

A formula $\Box (A \rhd B)$ can only be false at some world with at least
two successors. Thus, in $M$, we only need to consider the point $a$.
So, supppose $A\rhd \Diamond B$. For which $x$ with $aRx$ can we have
$x\Vdash A$? 

As we have to be able to go via an $S_x$-transition to 
a world where $\Diamond B$ holds, the only candidates for $x$ are 
$b,c$ and $d$. But clearly, $c$ and $f$ make true the same modal formulas.
From $f$ it is impossible to go to a world where $\Diamond B$ holds.

Thus, if $a\Vdash A\rhd \Diamond B$, the $A$ can only hold at $b$ or at 
$d$. But this automatically implies that $a\Vdash \Box (A\rhd B)$
and $M\models \principel{P_0}$.

It is not hard to see that $a\not \Vdash \principel{R}$.
Clearly, $a\Vdash p\rhd q$ and $b\Vdash \neg (p\rhd \neg r)$.
However, $d\not \Vdash q \wedge \Box r$ and thus
$a\not \Vdash \neg (p\rhd \neg r)\rhd q \wedge \Box r$.
\end{proof}

The following lemma tells us that \ilr is a proper extension 
of \extil{M_0P_0}.

\begin{lemma}
$\ilr \vdash \principel{M_0}, \principel{P_0}$
\end{lemma}

\begin{proof}
As 
$\il \vdash \Diamond A \wedge \Box C \rightarrow \neg (A \rhd \neg C)$
we get that 
$A \rhd B \rightarrow \Diamond A \wedge \Box C \rhd \neg (A \rhd \neg C)$
and \principel{M_0} follows from \principel{R}.

The principle \principel{P_0} follows directly from \principel{R} by 
taking $C=\neg B$.
\end{proof}

We can consider the principle \principel{R^*} that can be seen, in 
a sense, as the union of \principel{W} and \principel{R}.
\[
\principel{R^*}:\ \ \ \ A\rhd B \rightarrow \neg (A \rhd \neg C)\rhd 
B \wedge \Box C \wedge \Box \neg A
\]
\begin{lemma}\label{lemm:homo2}
$\extil{RW}=\ilrstar$
\end{lemma}

\begin{proof}
$\supseteq: A\rhd B \rightarrow A \rhd B \wedge \Box \neg A
\rightarrow \neg (A \rhd \neg C) \rhd B \wedge \Box C \wedge \Box \neg A$.

$\subseteq: A \rhd B \rightarrow \neg (A \rhd \neg C) \rhd B \wedge \Box C \wedge \Box \neg
A \rhd B \wedge \Box C$; and if
$A \rhd B$, then $A \rhd B \rhd ((B \wedge \Box \neg A)\vee \Diamond A )\rhd
B \wedge \Box \neg A$, as 
$A\rhd B \rightarrow \neg ( A\rhd \bot ) \rhd B \wedge \Box \top \wedge \Box \neg A$.
\end{proof}

\section{Arithmetical soundness of \principel{R}}\label{subs:risvetsound}


Let us first recall Definition \ref{defi:ilall}, that is, the 
definition of the interpretability logic of all reasonable arithmetical
theories. We shall write \intl{All}. We defined \intl{All} to 
be the set of modal formulas that are interpretability principles
in any reasonable arithmetical theory. That is, the set of $\varphi$ for which
\[
\forall T \forall *\ T\vdash \varphi^*.
\]

In \cite{viss:prel88} \intl{All} was conjectured to be \ilw. In 
\cite{Vis91} this conjecture was falsified and strengthened to 
a new conjecture. There it was conjectured that \ilwstar, which is a 
proper extension of \ilw, is \intl{All}.

In 
\cite{joo98} 
it was proved that the logic 
\extil{W^*P_0} is a proper extension of
\ilwstar, and that \extil{W^*P_0} is a subsystem of \intl{All}. This 
falsified the conjecture from \cite{Vis91}.
In \cite{joo98} it is also conjectured that \extil{W^*P_0} is not 
the same as \intl{All}.

In \cite{joo:prol00}  it is conjectured that
\extil{W^*P_0}=\intl{All}. As we will see below we have that the logic 
\extil{RW} is a subsystem of \intl{All} and a proper extension of \extil{W^*P_0}.
This rejects the conjecture pronounced in \cite{joo:prol00}. With all this 
conjecturing and refuting of conjectures we are rather hesitant in proposing
as a new conjecture that $\extil{RW} = \intl{All}$.\footnote{In fact, we have strong evidence that actually $\extil{RW} \neq \intl{All}$.}


We shall now give the proof that the new principle \principel{R}
is arithmetically valid in all reasonable theories.
In the proof we shall employ some well-known arithmetical facts. We will 
now first briefly summarize these facts.

\begin{definition}
A definable $T$-cut is a formula $I(x)$ with one free variable, such that
$T\vdash I(0) \wedge \forall x\  (I(x)\rightarrow I(x+1))$. ${\sf Cut}
(\cdot )$ will denote  the function that assigns to the code of a 
formula $\varphi$, the code of the formula expressing that
$\varphi$ is a cut, that is, 
$\varphi (0) \wedge \forall x\ (\varphi (x) \rightarrow \varphi (x+1))$
(whenever $\varphi$ is of the right format). 
\end{definition}

The function
${\sf Cut}(\cdot )$ is a very easy function. It is certainly 
provably total in
${\sf I}\Delta_{0} + \Omega_1$. In this section we shall denote 
the translation of a formula $\varphi$ under an interpretation $j$
by $j(\varphi )$. If $I$ is a cut and $\varphi$ a formula, we shall 
by $\varphi^I$ denote the formula $\varphi$, where all the
quantifiers in $\varphi$ are relativized to the cut $I$.
The following lemma is mentioned (as an exercise) in \cite{pudl85}.
It is central to many arguments in the field of 
formalized interpretability.

\begin{lemma}[Pudl\'ak]\label{lemm:pudlak}
There exists a function $f$, provably total in ${\sf I}\Delta_{0} + \Omega_1$,
such that for any reasonable arithmetical theory $T$, the following holds.
\[
T\vdash j : \alpha \rhd_T \beta \rightarrow
[\Box_T {\sf Cut}(f(j)) \wedge \forall \, \sigma {\in}\Sigma_1 ! \ 
j : \alpha \wedge \sigma^{f(j)} \rhd_T \beta \wedge \sigma]
\]
\end{lemma}

Another fact from arithmetic that we shall need, is that we
can perform the Henkin construction using numbers from a cut.
This is expressed by the following lemma.

\begin{lemma}\label{lemm:sneehenkin}
For any reasonable arithmetical theory $T$ we have that
\[
T \vdash \Box_T ({\sf Cut}(I))\rightarrow \Diamond_T^I\alpha \rhd_T
\alpha .
\]
\end{lemma}

These two lemmas are enough to prove the arithmetical soundness of the 
principle $\sf R$. Note that the $j$, $I$, $\alpha$ and $\beta$ in
Lemma \ref{lemm:pudlak} en \ref{lemm:sneehenkin} are parameters and 
hence could be universally quantified within the theory.

\begin{theorem}[Soundness of ${\sf R}$]
For any reasonable arithmetical theory $T$ we have the following.
\[
T \vdash \alpha \rhd \beta \rightarrow 
\neg (\alpha \rhd \neg\gamma) \rhd \beta \wedge \Box\gamma
\]
\end{theorem}

\begin{proof}
Let $f$ denote the function from Lemma \ref{lemm:pudlak}. To prove our
theorem, we reason in $T$ and assume $\alpha \rhd \beta$. Thus, for 
some interpretation $j$ we have $j : \alpha \rhd \beta$.
We now claim that
\[
\neg (\alpha \rhd \neg\gamma)\rightarrow
\Diamond ( \alpha \wedge \Box^{f(j)}\gamma).  \ \ \ \ \ \ (+)
\]
Let us first see that this claim, indeed entails the result.
\[
\begin{array}{rcl}
\neg (\alpha \rhd \neg\gamma) 
& \rhd 
& \mbox{ By $(+)$}\\
\Diamond ( \alpha \wedge \Box^{f(j)}\gamma)
& \rhd 
& \mbox{ By $\sf J5$}\\
\alpha \wedge \Box^{f(j)}\gamma
& \rhd 
& \mbox{ By Lemma \ref{lemm:pudlak} and $j : \alpha \rhd \beta$}\\
\beta \wedge \Box\gamma
\end{array}
\]
Thus, now we only need to prove the claim. We will prove $(+)$ by 
showing the logical equivalent
\[
\Box (\alpha \rightarrow \Diamond^{f(j)} \neg\gamma)
\rightarrow \alpha \rhd \neg\gamma. \ \ \ \ \  (++)
\]
We reason as follows.
\[
\begin{array}{rcl}
\Box (\alpha \rightarrow \Diamond^{f(j)}\neg \gamma) 
& \rightarrow 
& \mbox{ By $\sf J1$}\\
\alpha \rhd \Diamond^{f(j)} \neg\gamma 
& \rightarrow 
& \mbox{ By Lemma \ref{lemm:sneehenkin} and $\sf J2$}\\
\alpha \rhd \neg\gamma
\end{array}
\]
\end{proof}

\bibliographystyle{plain}


%
%

\end{document}